\documentclass[11pt]{amsart}
\usepackage{amsmath,amssymb}
\usepackage{tikz}
\usepackage{todonotes}

\newtheorem{theorem}{Theorem}[section]
\newtheorem{lemma}[theorem]{Lemma}
\newtheorem{corollary}[theorem]{Corollary}

\newtheorem*{theorem*}{Theorem}

\theoremstyle{definition}
\newtheorem{definition}[theorem]{Definition}

\numberwithin{equation}{section}

\newcommand{\R}{\mathbb{R}}

\newcommand{\pathto}{\rightsquigarrow}

\newcommand{\Gr}{{\rm Gr}}
\renewcommand{\int}{\mathop{\rm int}}

\newcommand{\sgn}{\mathop{\rm sgn}}

\newcommand{\Net}{{\rm Net}}
\newcommand{\wt}{\mathop{\rm wt}}

\newcommand{\var}{\mathop{\rm var}}
\newcommand{\link}{\mathop{\rm lk}\nolimits}
\newcommand{\chark}{\mathop{\rm char}(\k)}
\renewcommand{\k}{\mathbf{k}}

\begin{document}

\title[Boundary measurement and sign variation ]{Boundary measurement and sign variation in real projective space}

\author{John Machacek}\address[Machacek]
{Department of Mathematics and Statistics\\ York  University\\ To\-ron\-to, Ontario M3J 1P3\\ CANADA}
\email{machacek@yorku.ca}
\urladdr{}
\thanks{The author acknowledges generous support of the York Science Fellowship.}


\begin{abstract}
We define two generalizations of the totally nonnegative Grassmannian and determine their topology in the case of real projective space.
We find the spaces to be PL manifolds with boundary which are homotopy equivalent to another real projective space of smaller dimension.
One generalization makes use of sign variation while the other uses boundary measurement.
Spaces arising from boundary measurement are shown to admit Cohen-Macaulay triangulations.
\end{abstract}

\maketitle

\section{Introduction}

The totally nonnegative Grassmannian is a certain subset of the real Grassmannian introduced by Postnikov~\cite{Pos06}.
This subset of the Grassmannian has since been the subject of much research due its rich structure and appearance in other areas including the computation of scattering amplitudes in physics~\cite{scatter} through its connection to the amplituhedron~\cite{amp}.
We define two generalizations of the totally nonnegative Grassmannian both of which build upon work related to the amplituhedron.
One generalization uses a construction called boundary measurement which was originally defined on a disk~\cite{Pos06} and has since been defined for any closed orientable surface with boundary in the scattering amplitudes literature~\cite{FGM14,FGPW15}.
The other generalization uses sign variation which has been previously used in studying the totally nonnegative Grassmannian and the amplituhedron~\cite{Karp, KarpWilliams, Unwinding}.

An important problem in total positivity is to understand the topology of the spaces arising.
Galashin, Karp, and Lam have recently settled an important conjecture in the area by showing that the totally nonnegative Grassmannian (and more generally the totally nonnegative part of any partial flag manifold as defined by Lusztig~\cite{Lus, Lus98, Luspartial}) is a regular CW complex homeomorphic to a closed ball~\cite{GKL3}.
This result was conjectured by Postnikov~\cite{Pos06} for the Grassmannian and by Williams for partial flag varieties~\cite{WilShelling}.
There was a large body of work producing evidence for these former conjectures and establishing various properties~\cite{RW08,PSW, RWregular, GKL1, GKLadv}.
The main theorem of~\cite{GKL3} also shows that the link of the identity element inside the totally positive part of the unipotent radical of an algebraic group is a regular cell complex.
This result was conjectured by Fomin and Shapiro~\cite{FS2000} and originally proven by Hersh~\cite{Hersh}.

Spaces in total positivity have rich structure.
The totally nonnegative part also turns out to be remarkably simple topologically compared to the whole space.
Our spaces relax total positivity in a controlled way.
The topology of our spaces will be completely determined in the special case of real projective space.
We will find that we never have a closed ball, except in the case of total positivity.
However, our spaces will have desirable properties including being PL manifolds (with boundary) and in certain cases admit Cohen-Macaulay triangulations.

This paper is structured as follows.
In Section~\ref{sec:gen} we will define the two generalizations of the totally nonnegative Grassmannian.
In the remaining sections we will restrict our attention to the special case of real projective space where we already observe interesting phenomena.
Much of our work comes from using sign variation to define a regular CW complex in this special case.
Mn\"ev's universality theorem~\cite{Mnev} implies the matroid stratification~\cite{GGMS} of the Grassmannian is necessarily complicated.
Due to this we do not currently know how to define cells in general Grassmannians compatible with sign variation. 
In Section~\ref{sec:var} we analyze the CW structure along with its closure poset and order complex.
In Theorem~\ref{thm:manifold} we show these spaces are PL manifolds (with boundary) by making use of Bj\"{o}rner-Wachs lexicographic shellability~\cite{BW} to obtain order complexes which are combinatorial manifolds. 
We determine homotopy types in Theorem~\ref{thm:collaspe} through the use of Forman's discrete Morse theory~\cite{Forman}.
This allows us to deduce when we have a Cohen-Macaulay order complex in Corollary~\ref{cor:CM}.
In Section~\ref{sec:meas} we consider boundary measurement.
We find in Theorem~\ref{thm:onesource} that our generalization in terms of boundary measurement is a special case of the generalization in terms of sign variation.
Moreover, the boundary measurement generalized spaces are exactly the cases in which we have a Cohen-Macaulay order complex.

\subsection*{Acknowledgments}
The author thanks Steven Karp and Michael Shapiro for discussions and comments during the development of this paper.
The author also wishes to thank the anonymous referee for their careful reading and comments.

\section{The totally nonnegative Grassmannian generalizations}
\label{sec:gen}
In this section we define two generalizations of the totally nonnegative Grassmannian.
As notation, for any positive integer we define $[n] := \{1,2,\dots, n\}$ and let $\binom{[n]}{k}$ denote the collection of all $k$-element subsets of $[n]$.

\subsection{The Grassmannian and sign variation}
We briefly review the Grassmainian from the point of view we will use.
The \emph{(real) Grassmannian} $\Gr_{k,n}$ is the set of subspaces $V \leq \mathbb{R}^n$ such that $\dim V = k$.
Each $V \in \Gr_{k,n}$ can be represented by a full rank $k \times n$ matrix whose rows span the subspace $V$.
Let $A$ be an such matrix representing $V \in \Gr_{k,n}$, then for any $I \in \binom{[n]}{k}$ let $\Delta_I(V)$ be the maximal minor of $A$ with columns indexed by $I$.
Each $\Delta_I(V)$ is called a \emph{Pl\"ucker coordinate}.
The  Pl\"ucker coordinates of $V$ depend on the choice of representing matrix $A$, but they are well-defined up to simultaneous scaling by a nonzero constant.
Hence, the Pl\"ucker coordinates are well-defined as elements of the real projective space $\mathbb{RP}^{\binom{n}{k} - 1}$.
A subspace $V \in \Gr_{k,n}$ is called \emph{totally nonnegative} if there is a scaling such that $\Delta_I(V) \geq 0$ for all $I \in \binom{[n]}{k}$ and \emph{totally positive} if $\Delta_I(V) > 0$ for all $I \in \binom{[n]}{k}$.
Following Postnikov~\cite{Pos06} we define the \emph{totally nonnegative Grassmannian} and \emph{totally positive Grassmannian} to be
\[\Gr^{\geq 0}_{k,n} := \{ V \in \Gr_{k,n} : V \text{ is totally nonnegative}\}\]
and 
\[\Gr^{>0}_{k,n} := \{ V \in \Gr_{k,n} : V \text{ is totally positive}\}\]
respectively.
This notion of positivity agrees with the Grassmannian cases of Lusztig's notion of positivity in partial flag varieties.

We define the \emph{sign function} by
\[\sgn(x) := \begin{cases} - & \text{if } x < 0\\ 0 & \text{if } x = 0\\ + & \text{if } x > 0\end{cases}\]
for any $x \in \mathbb{R}$.
We extend this function to vectors by
\[\sgn(v) := (\sgn(v_1), \sgn(v_2), \dots, \sgn(v_n))\]
for any $v \in \mathbb{R}^n$.
We will consider $\sgn(v)$ up to projective equivalence by identifying $\sgn(v)$ with $\sgn(\lambda v)$ for any $\lambda \in \mathbb{R} \setminus \{0\}$.
For example, we have
\[\sgn((-1,2,0,4,-3)) = (-,+,0,+,-) = (+,-,0,-,+)\]
and in this way can always assume that the first nonzero entry of $\sgn(v)$ is $+$.
If $v$ is any vector or sign vector, then the \emph{weight} of $v$ is denoted $\wt(v)$ and defined to be the number of nonzero entries in $v$.

Given a vector $v \in \mathbb{R}^n$ the \emph{sign variation} of $v$ is denoted $\var(v)$ and is the number of times $v$ changes sign where zeros are ignored.
The sign variation of $v$ can be computed from $\sgn(v)$.
We also define
\[\overline{\var}(v) := \max \{\var(w) : w \in \mathbb{R}^n \text{ and } w_i = v_i \text{ whenever } v_i \neq 0\}\]
which gives the largest possible sign variation when zeros are allows to replaced with any real number.
For example, $\var((1,0,3,-1,2)) = 2$ and $\overline{\var}((1,0,3,-1,2)) = 4$.
We have the following description of the totally nonnegative Grassmannian 
\[\Gr^{\geq 0}_{k,n} = \{V \in \Gr_{k,n} : \var(v) \leq k-1 \text{ for all } v \in V\}\]
and totally positive Grassmannian 
\[\Gr^{>0}_{k,n} = \{V \in \Gr_{k,n} : \overline{\var}(v) \leq k-1 \text{ for all } v \in V \setminus \{0\}\}\]
in terms of sign variation~\cite{GK, SW}.

\begin{definition}[Bounded sign variation Grassmannian]
For any $k-1 \leq m \leq n-1$, we then define the set
\[\Gr^{\var \leq m}_{k,n} := \{V \in \Gr_{k,n} : \var(v) \leq m \text{ for all } v \in V\}\]
which we call a \emph{bounded sign variation Grassmannian}.
\label{def:bounded}
\end{definition}
In Definition~\ref{def:bounded} there is no loss of generality assuming $k-1 \leq m \leq n-1$.
If $m < k-1$,  then $\Gr^{\var \leq m}_{k,n} = \emptyset$.
This can be seen by putting a matrix representing $V \in \Gr_{k,n}$ into reduced row echelon form.
Taking the alternating sum of the vectors which are the rows of the  reduced row echelon form gives an element $v \in V$ with $\var(v) \geq k-1$.
Also, since $\var(v) \leq n-1$ for any $v \in \mathbb{R}^n$ is follows that $\Gr^{\var \leq m}_{k,n} = \Gr_{k,n}$ whenever $m \geq n-1$.

We now recall a few results of Karp phrased in terms of the bounded sign variation Grassmannian which further demonstrate the analogy between $\Gr^{\geq 0}_{k,n}$ and $\Gr^{\var \leq m}_{k,n}$.
A subspace $V \in \Gr_{k,n}$ is called \emph{generic} if all Pl\"ucker coordinates are nonzero.
The set of generic elements of $\Gr^{\geq 0}_{k,n}$ is exactly $\Gr^{> 0}_{k,n}$, and the totally positive Grassmannian is dense in the totally nonnegative Grassmannian.

\begin{theorem*}[{\cite[Theorem 1.2 (i)]{Karp}}]
If $V \in \Gr^{\var \leq m}_{k,n}$, then for each $v \in V$ 
\[\var((\Delta_{I \cup \{i\}}(V))_{i \in [n] \setminus I}) \leq m - k +1\] 
for all $I \in \binom{[n]}{k-1}$.
\end{theorem*}

\begin{theorem*}[{\cite[Theorem 1.4]{Karp}.}]
Generic elements of $\Gr_{k,n}^{\var \leq m}$ form a dense subset.
\end{theorem*}

\subsection{Boundary measurement}
Let $N = (V,E)$ be a directed graph with finite vertex set $V$ and finite edge set $E$.
Each edge $e \in E$ is assigned a \emph{weight} $x_e$.
We will sometimes consider $x_e$ as a formal variable and work in $\R[[x_e : e \in E]]$ the ring of formal power series in the variables $\{x_e\}_{e \in E}$ with coefficients in $\R$.
All formal power series we consider will have rational expressions.
We will also consider specializations of these rational expressions where each $x_e$ takes a positive real value.
As in~\cite{Pos06}, we will use the term \emph{directed network} to refer to the directed graph $N = (V,E)$ along with edge weights $\{x_e\}_{e \in E}$.

Let $S$ be a closed orientable surface of genus zero with $b > 0$ boundary components.
Let $\Net^S_{k,n}$ denote the collection of directed networks $N$ embedded on $S$ such that
\begin{enumerate}
    \item[(i)] $n$ vertices are on the boundary of $S$,
    \item[(ii)] $k$ of the $n$ boundary vertices are nonisolated sources,
    \item[(iii)] the remaining $n-k$ boundary vertices are isolated or univalent sinks,
    \item[(iv)] each interior vertex is trivalent,
    \item[(v)] and each interior vertex is neither a source nor a sink.
\end{enumerate}
Elements of $\Net^S_{k,n}$ are considered up to isotopy.
Interior vertices then come in two types.
A \emph{white} interior vertex has one incoming edge and two outgoing edges while a \emph{black} interior vertex has one outgoing edge and two incoming edges.
An example directed network in $\Net^S_{2,4}$ where $S$ is the annulus is shown in Figure~\ref{fig:network}.

\begin{figure}
\begin{tikzpicture}
\draw[thick] (0,0) circle (3cm);
\draw[thick] (0,0) circle (1cm);
\draw[dotted] (0,3) to (0,1);

\node[draw, circle, scale=0.5](a)  at (-1.41,1.41) {};
\node[draw, circle, scale=0.5, fill=black](b)  at (-1.41,-1.41) {};
\node[draw, circle, scale=0.5](c)  at (1.41,-1.41) {};
\node[draw, circle, scale=0.5, fill=black](d)  at (1.41,1.41) {};

\draw[-{latex}] (-2.12,2.12) to (a);
\draw[-{latex}] (a) to (d);
\draw[-{latex}] (a) to (b);
\draw[-{latex}] (b) to (-2.12,-2.12);
\draw[-{latex}] (c) to (b);
\draw[-{latex}] (2.12,-2.12) to (c);
\draw[-{latex}] (c)to (d);
\draw[-{latex}] (d) to (0.707,0.707);

\node at (-2.24,2.24) {$1$};
\node at (-2.24,-2.24) {$2$};
\node at (2.24,-2.24) {$3$};
\node at (0.55,0.55) {$4$};

\node at (-1.7,2) {$x_1$};
\node at (-1.6,-1.9) {$x_2$};
\node at (1.6,-1.9) {$x_3$};
\node at (0.85,1.15) {$x_4$};

\node at (-0.3, 1.56) {$x_5$};
\node at (-1.6,0) {$x_6$};
\node at (0, -1.56) {$x_7$};
\node at (1.62,0) {$x_8$};

\end{tikzpicture}
\caption{A directed network on the annulus.}
\label{fig:network}
\end{figure}
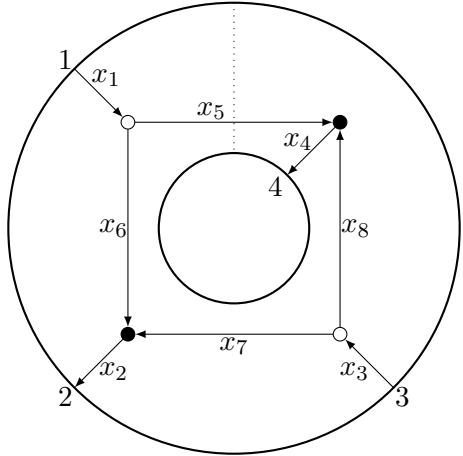

We make $b-1$ cuts between various boundary components so that the complement of these cuts is simply connected.
Let $T$ denote the complement of the cuts in $S$.
We will refer to the pair $(S,T)$ as a \emph{surface with chosen cuts} and to $T$ as a \emph{choice of cuts}.
The following construction depends the choice of cuts.
The construction will also make use of a drawing of $S$ in the plane.
The author has previously shown the construction is independent of how $S$ is drawn in the plane~\cite[Theorem 4]{Meas}.

The boundary $\partial T$ is homeomorphic to a circle and can be assumed to be piecewise smooth.
We will consider the boundary equipped with a piecewise smooth parameterization $\phi:[0,1] \to \partial T$ with $\phi(0) = \phi(1)$.
We assume throughout that all parameterizations are piecewise smooth and have nowhere zero derivative.
Traversing $\partial T$ according to $\phi$ induces an ordering the boundary vertices of any $N \in \Net^S_{k,n}$.
We will always label the boundary vertices $\{1,2,\dots, k\}$.

Given any smooth closed curve $C$ in the plane let $\mathbf{T}: C \to S^1$ give the unit tangent vector of each point. 
The \emph{rotation number} of $C$ is the degree of the map $\mathbf{T} \circ \psi: S^1 \to S^1$ where $\psi: S^1 \to C$ is a parameterization of $C$.
The rotation number of a continuous curve $C$ we will be the rotation number of a smooth curve approximating $C$.
The choice of which smooth curve is used as an approximation does not affect the rotation number.

We consider $S$ drawn in the plane by choosing a boundary component of $S$ to be external.
The external boundary component bounds a disk in the plane.
The rest of $S$ as well as any directed network on $S$ can then be draw inside this disk.
Consider a network $N \in \Net^S_{k,n}$ drawn in the plane
Then overlay the cuts used to construct $T$.
For any path $P:i \pathto j$ between boundary vertices $i$ and $j$ we form a closed curve $C(P)$ in the plane as follows
\begin{enumerate}
\item Traverse the path $P$ from $i$ to $j$ in $S$.
\item Follow the boundary of $T$ in the direction specified by $\phi$ from $j$ to $i$.
\end{enumerate}

The \emph{boundary measurement matrix}  is the $k  \times n$ matrix $B(N, T)$ with entries given by
\[B_{ij} = \sum_{P: i \pathto j} (-1)^{s_{ij} + r_P + 1} \wt(P)\]
where $i$ is a source and $j$ is any boundary vertex.
Here $s_{ij}$ denotes the number of sources on the boundary strictly between $i$ and $j$, and $r_P$ denotes the rotation number of $C(P)$.

Postnikov~\cite{Pos06} gave the original definition of the boundary measurement matrix in the case where the surface is a disk.
The boundary measurement matrix was considered for networks on the annulus by Gekhtman, Shapiro, and Vainshtein~\cite{GSV08} and for networks on any closed orientable genus zero surface with boundary by Franco, Galloni, and Mariotti~\cite{FGM14}.
Boundary measurement has also been defined by Franco, Galloni, Penante, and Wen~\cite{FGPW15} on any  closed orientable surface with boundary.
Further study of this more general boundary measurement matrix is done in~\cite{Meas}.
Letting $N$ be the directed network in Figure~\ref{fig:network} with choice of cuts $T$ shown by the dotted line, we find that
\[B(N,T) = \begin{bmatrix} 1 & x_1x_6x_2 & 0 & x_1x_5x_4 \\ 0 & x_3x_7x_2 & 1 & x_3x_8x_4\end{bmatrix}\]
is the boundary measurement matrix.

\begin{definition}[Boundary mesurement Grassmannian]
Given a closed orientable surface of genus zero with boundary $S$ and $k \leq n$, we define a \emph{boundary measurement Grassmannian} as
\[\Gr^S_{k,n} := \{B(N,T) : N \in \Net^S_{k,n} \text{ and } (S,T) \text{ is a surface with chosen cuts}\}\]
where we think of the full rank $k \times n$ matrix $B(N,T)$ as an element of $\Gr_{k,n}$.
\end{definition}

Postnikov~\cite{Pos06} defined boundary measurement because when $S$ is the disk we have $\Gr^S_{k,n} = \Gr^{\geq 0}_{k,n}$.
Hence, the boundary measurement Grassmannian also gives a generalization of the totally nonnegative Grassmannian.
Moreover, the author has shown the Pl\"ucker coordinates of $B(N,T)$ for $N \in \Net^S_{k,n}$ have a formula which is a rational expression described in terms of nonintersecting paths~\cite[Corollary 9]{Meas}.
Such a formula when the surface is a disk was originally shown by Talaska~\cite{Talaska}.

\section{Complexes inside real projective space}
\label{sec:var}

In this section we focus on the case where $k = 1$.
So, we will be working inside $\Gr_{1,n}$ which is the real projective space $\mathbb{RP}^{n-1}$.
A regular CW structure on $\mathbb{RP}^{n-1}$ will be presented which is compatible with sign variation in a way that will allow use to obtain a CW structure on $\Gr^{\var \leq m}_{1,n}$ for any $0 \leq m \leq n-1$.
We will investigate the topology of $\Gr^{\var \leq m}_{1,n}$ using these CW complexes along with their closure posets and corresponding order complexes.
The next subsection contains standard definitions and results in combinatorial topology.

\subsection{Some general combinatorial and poset topology}
Here we give a quick overview of some combinatorial topology with a bias towards what is needed for our purposes.
For further references one can see the survey of Bj\"{o}rner~\cite{surveyB} and the survey of Wachs~\cite{surveyW}.

An \emph{(abstract) simplicial complex} $\Delta$ on a (finite) vertex set $V$ is a collection of nonempty subsets of $V$ that is closed under inclusion.
Every simplicial complex has a \emph{geometric realization} denoted $\|\Delta\|$ which is a topological space.
An element $F \in \Delta$ is called a \emph{face}, and the dimension of $F \in \Delta$ is $\dim(F) = |F|  - 1$.
The faces of $\Delta$ which are maximal with respect to inclusion are call \emph{facets}.
A simplicial complex is said to be \emph{pure} if all facets have the same dimension.
A simplicial complex is \emph{shellable} if there is an ordering of the facets $F_1, F_2, \dots, F_k$ such that $\left(\bigcup_{i = 1}^{j-1} F_i \right) \cap F_j$ is pure and  $(\dim(F_j) - 1)$-dimensional for all $2 \leq j \leq k$.
The \emph{join} of two simplicial complexes $\Delta$ and $\Gamma$ on disjoint vertex sets is the simplicial complex 
\[\Delta * \Gamma := \Delta \cup \Gamma \cup \{F \cup G : F \in \Delta, G \in \Gamma\}.\]
The join of two simplicial complexes is shellable if and only if both of the simplicial complexes are themselves shellable.

The \emph{link} of any face $F \in \Delta$ is denoted $\link_{\Delta}(F)$ defined by 
\[\link_{\Delta}(F) := \{G: F \cup G \in \Delta, F \cap G = \varnothing, G \neq \varnothing\}.\]
We also let $\link_{\Delta}(\varnothing) = \Delta$.
We call $\Delta$ a \emph{normal $d$-pseudomanifold} provided
\begin{enumerate}
    \item[(NP1)] $\Delta$ is pure of dimension $d$,
    \item[(NP2)] every $(d-1)$-dimensional face of $\Delta$ is contained in at most two facets of $\Delta$,
    \item[(NP3)] and for all $F \in \Delta$ with $\dim F \leq d-2$ the link $\link_{\Delta} (F)$ is connected.
\end{enumerate}
The term \emph{$d$-pseudomanifold} refers to a complex $\Delta$ satisfying only the conditions (NP1) and (NP2).
We may say just normal pseudomanifold or pseudomanifold when we do not wish to specify the dimension.
In our terminology all normal pseudomanifolds and pseudomanifolds are allowed to have possibly nonempty boundary.
One fact we will need is that the link of a face of codimension at least $2$ in a normal pseudomanifold is a pseudomanifold.
Another fact we use later is that if a pseudomanifold $\Delta$ is shellable, then $\|\Delta\|$ is either a PL-sphere or a PL-ball.
A PL-sphere or PL-ball is a simplicial complex combinatorially equivalent to a subdivison of a simplex or boundary of a simplex respectively.
More on PL topology can be found in the book~\cite{Hudson}.

A simplicial complex $\Delta$ is called a \emph{combinatorial $d$-manifold} if for every face $F \in \Delta$ we have that $\|\link_{\Delta}(F)\|$ is a $(d- \dim(F) -1)$-dimensional PL sphere or PL ball.
Similar to before we just say combinatorial manifold when specifying the dimension is not needed.
If $\Delta$ is a combinatorial manifold, then $\| \Delta \|$ is a PL-manifold.
Again we by default allow for possibly nonempty boundary.

Let $P$ be a finite poset.
An element $y \in P$ is said to \emph{cover} $x \in P$ if $x < y$ and there does not exist any $z \in P$ such that $x < z < y$.
In the case $y$ covers $x$ we write $x \lessdot y$.
A sequence of elements $x_0 < x_1 < \cdots < x_t$ is a \emph{chain} of length $t$.
For any $x,y \in P$ we define the \emph{closed interval} $[x,y]$ by
\[[x,y] := \{z \in P : x \leq z \leq y\}\]
and the \emph{open interval} $(x,y)$ by
\[(x,y) := \{z \in P : x < z < y\}.\]
The \emph{length} of $P$ is the maximum of the lengths of all chains in $P$.
We call a poset \emph{pure} if all chains which are maximal by inclusion have the same length.

The \emph{order complex} of $P$ is denoted $\Delta(P)$ and is the simplicial complex whose vertices of the elements of $P$ and whose $k$-dimensional faces are chains $x_0 < x_1 < \cdots < x_k$.
In the case $P$ has a unique minimal or a unique maximal element, we denote it by $\hat{0}$ or $\hat{1}$ respectively.
Elements which cover $\hat{0}$ are called \emph{atoms}, and elements covered by $\hat{1}$ are called \emph{coatoms}.
 If $P$ has $\hat{0}$ and every interval $[\hat{0},x]$ for $x \in P$ is isomorphic to a Boolean lattice we say that $P$ is a \emph{simplicial poset}.
 The \emph{join} of two posets $P$ and $Q$ is the poset $P * Q$ whose underlying set in the disjoint union of $P$ and $Q$ and whose order relation includes all relations from $P$ and $Q$ along with declaring $x < y$ for all $x \in P$ and $y \in Q$.
It follows that $\Delta(P*Q) = \Delta(P) * \Delta(Q)$.
 
Let $P$ be a \emph{bounded poset} which means $P$ has $\hat{0}$ and $\hat{1}$.
We say $P$ admits a \emph{recursive coatom ordering}~\cite{BW} if the length of $P$ is $1$, or else if $P$ has an ordering $c_1, c_2, \dots, c_t$ of the coatoms of $P$ satisfying
\begin{enumerate}
\item[(RCO1)] for all $1 \leq j \leq t$ the interval $[\hat{0}, c_j]$ admits a recursive coatom ordering in which coatms of $[\hat{0}, c_j]$ which belong to $[\hat{0}, c_i]$ for some $i < j$ come first,
\item[(RCO2)] and for all $1 \leq i < j \leq t$ if $x < c_i, c_j$, then there exists $k < j$ and a coatom $d$ of $[\hat{0}, c_j]$ such that $x \leq d < c_k$.
\end{enumerate}
If $P$ admits a recursive coatom ordering, then $\Delta(P)$ is shellable.
In the case $\Delta(P)$ is shellable for a poset $P$ we will say that $P$ is shellable.

We let $\overline{P}$ denote the \emph{proper part} of $P$ which is obtained by removing $\hat{0}$ and $\hat{1}$ if they are present.
A bounded poset $P$ is shellable if and only if $\overline{P}$ is shellable.
If a poset is shellable, then every open or closed interval is also shellable.

 For a CW complex $X$ we let $\mathcal{F}(X)$ denote the \emph{closure poset} of $X$ which consists of the closure of cells ordered by inclusion.
For a regular CW complex $X$ it is the case that $X$ is homeomorphic to $\Delta(\mathcal{F}(X))$.
In fact, $\Delta(\mathcal{F}(X))$ will be the barycenteric subdivision of $X$.

\subsection{The CW structure}
We now return to our study of $\Gr^{\var \leq m}_{1,n}$ and will define a regular CW decomposition of $\Gr^{\var \leq m}_{1,n}$.
Given a sign vector $\omega \in \{-,0,+\}^n \setminus \{0\}^n$, the set
\[U_{\omega} := \{v \in \mathbb{RP}^{n-1} : \sgn(v) = \omega\}\]
is homeomorphic to an open ball.
The closure of this open ball is the closed ball which is a  $(\wt(\omega)-1)$-simplex consisting of all $U_{\omega'}$ where $\omega'$ is obtain from $\omega$ by replacing some of its nonzero entries by zero.
The sets $U_{\omega}$ give a decomposition which makes $\mathbb{RP}^{n-1}$ into a regular CW complex.
The subcomplex given $\{U_{\omega} : \var(\omega) \leq m\}$ defines a regular CW structure on $\Gr^{\var \leq m}_{1,n}$.
We will let $P_{n,m}$ denote the closure poset of this regular CW complex we have just defined on  $\Gr^{\var \leq m}_{1,n}$.
The Hasse diagram of the poset $P_{3,2}$ is shown in Figure~\ref{fig:P_32}.
We then let $\Delta_{n,m}$ denote $\Delta(P_{n,m})$.
An element of $P_{n,m}$ is a cell $U_{\omega}$ which we will typically just identify with the sign vector $\omega$.
We have the following lemma.

\begin{lemma}
For any $0 \leq m \leq n-1$ the poset $P_{n,m} \cup \{\hat{0}\}$ is a pure simplicial poset.
\label{lem:simplicialPoset}
\end{lemma}
\begin{proof}
Each maximal chain in $P_{n,m} \cup \{\hat{0}\}$ will have length $n$ since each such chain will be of the form
\[\hat{0} \lessdot \omega_1 \lessdot \cdots \lessdot \omega_n\]
where $\wt(\omega_i) = i$ for $1 \leq i \leq n$.
Thus, $P_{n,m} \cup \{\hat{0}\}$ is pure.
Furthermore, we see that for any nonzero sign vector $\omega$ and interval $[\hat{0}, \omega]$ is a Boolean lattice of rank $\wt(\omega)$.
\end{proof}

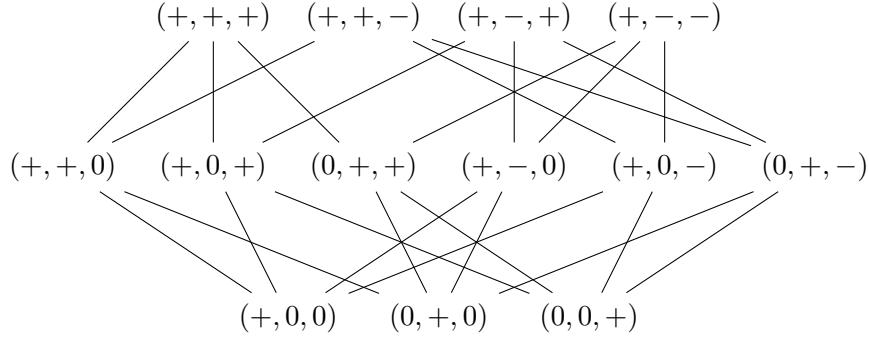
\begin{figure}
    \centering
    \begin{tikzpicture}
    \node (a) at (-3,2) {$(+,+,+)$};
    \node (b) at (-1,2) {$(+,+,-)$};
    \node (c) at (1,2) {$(+,-,+)$};
    \node (d) at (3,2) {$(+,-,-)$};
    
    \node (e) at (-5,0) {$(+,+,0)$};
    \node (f) at (-3,0) {$(+,0,+)$};
    \node (g) at (-1,0) {$(0,+,+)$};
    \node (h) at (1,0) {$(+,-,0)$};
    \node (i) at (3,0) {$(+,0,-)$};
    \node (j) at (5,0) {$(0,+,-)$};
    
    \node (k) at (-2,-2) {$(+,0,0)$};
    \node (l) at (0,-2) {$(0,+,0)$};
    \node (m) at (2,-2) {$(0,0,+)$};
    
    \draw (a) to (e);
    \draw (a) to (f);
    \draw (a) to (g);
    
    \draw (b) to (j);
    \draw (b) to (i);
    \draw (b) to (e);
    
    \draw (c) to (h);
    \draw (c) to (f);
    \draw (c) to (j);
    
    \draw (d) to (h);
    \draw (d) to (i);
    \draw (d) to (g);

    \draw (k) to (e);
    \draw (k) to (f);
    \draw (k) to (h);
    \draw (k) to (i);
    
    \draw (l) to (e);
    \draw (l) to (g);
    \draw (l) to (h);
    \draw (l) to (j);
    
    \draw (m) to (f);
    \draw (m) to (g);
    \draw (m) to (i);
    \draw (m) to (j);

    \end{tikzpicture}
    \caption{The Hasse diagram of the poset $P_{3,2}$.}
    \label{fig:P_32}
\end{figure}

\subsection{Combinatorial and PL manifolds}

In this subsection we will show that the simplicial complex $\Delta_{n,m}$ is a combinatorial $(n-1)$-manifold and hence $\Gr^{\var \leq m}_{1,n}$ is a PL manifold.

\begin{lemma}
For any $0 \leq m \leq n-1$ and $F \in \Delta_{n,m}$ the link $\link_{\Delta_{n,m}}(F)$ is shellable.
\label{lem:shellable}
\end{lemma}
\begin{proof}
For any $x,y \in P_{n,m} \cup \{\hat{0}\}$ we find by Lemma~\ref{lem:simplicialPoset} that the open interval $(x,y)$ is isomorphic to a Boolean lattice with top and bottom elements removed.
It follows that $\Delta((x,y))$ is shellable.
By considering joins, we may then reduce to the case for an upset $\{y : x < y\} \subseteq P_{n,m}$ for some $x \in P_{n,m}$.
Moreover, by using that open intervals of shellable posets are shellable we may look at such an upset where $x = \sgn(e_i)$ for some $1 \leq i \leq n$ where $e_i$ is the $i$th standard basis vector.
We will consider
\[Q_i = \{y : \sgn(e_i) \leq y\} \subseteq P_{n,m} \cup \{\hat{1}\}\]
which is a bounded poset with bottom element $\sgn(e_i)$ and top element $\hat{1}$.
So, we may give recursive coatom ordering for $Q_i$ and conclude that $\overline{Q}_i$ is shellable.
Showing this will prove the lemma.

Let the coatoms of $Q_i$ be $c_1, c_2, \dots, c_t$.
We will later explain the ordering of the coatoms, but first notice that each interval $[\sgn(e_i), c_j]$ for $1 \leq j \leq t$ is a Boolean lattice.
Any coatom ordering in a Boolean lattice is a recursive coatom ordering, and it follows that (RCO1) holds no matter what coatom ordering we choose in $Q_i$.

Recalling that $Q_i$ includes $\hat{1}$, the coatoms are sign vectors $\omega$ such that $\wt(\omega) = n$ and $\var(\omega) \leq m$.
We have a bijection $\phi_i$ between sign vectors $\omega$ such that $\wt(\omega) = n$ with $\var(\omega) = \ell$ and $\ell$ element subsets of $[n] \setminus \{i\}$ which we will now describe.
Consider such a sign vector $\omega$.
For $j < i$ we have $j \in \phi_i(\omega)$ if and only if $\omega_j$ and $\omega_{j+1}$ differ in sign.
For $j > i$ we have $j \in \phi_i(\omega)$ if and only if $\omega_j$ and $\omega_{j-1}$ differ in sign.
So, $\phi_i$ records all sign changes with the convention that whenever two adjacent entries differ in sign the entry furthest from $i$ is recorded.
Now coatoms of $Q_i$ are in bijective correspondence with subsets of $[n] \setminus \{i\}$ of cardinality at most $m$.

We define an order $<_i$ on $[n] \setminus \{i\}$ by $i + a <_i i + b$ if $|a| > |b|$ or if $|a| = |b|$ and $a > 0 > b$.
So, the smallest elements in this order are the elements furthest from $i$.
The order $<_i$ is then extended to all subsets of $[n] \setminus \{i\}$ in a graded lexicographical fashion.
That is, given $A,B \subseteq [n] \setminus \{i\}$ we have $A <_i B$ if and only if $|A| < |B|$ or $|A| = |B|$ while $A$ is lexicographically less than $B$ with respect to $<_i$.
We claim that $<_i$ gives a recursive coatom ordering in $Q_i$.
So, the coatoms of $Q_i$ are ordered $c_1, c_2, \dots, c_t$ where $j_1 < j_2$ if and only if $\phi_i(c_{j_1}) <_i \phi_i(c_{j_2})$.

We must verify (RCO2).
Take $1 \leq j_1 < j_2 \leq t$ and $x < c_{j_1}, c_{j_2}$.
Let $x = (x_1, x_2, \dots, x_n)$ and set $F = \{a : x_a = 0\}$ which records the positions that the sign vector $x$ is $0$.
Notice changing values of a sign vector in positions from the set $F$ does not change whether or not a sign vector is above $x$.
If $F \cap \phi_i(c_{j_2}) = \varnothing$, then it follows that $c_{j_2}$ is smallest coatom in our coatom ordering which is above $x$.
Indeed in the case $F \cap \phi_i(c_{j_2}) = \varnothing$ we can obtain $c_{j_2}$ from $x$ as follows.
Start at entry $x_i$ which is nonzero since $\sgn(e_i) \leq x$.
Then move away from $i$ (in either direction).
Whenever an entry which is zero is encountered, it is to be filled with whatever is present in the entry directly before it.
This means the number of sign flips is as small as possible.
Also, any sign flips are as far from $i$ as possible.
Hence, $\phi_i(c_{j_2})$ is least among all coatoms above $x$.
However, this contradicts $x < c_{j_1}, c_{j_2}$ with $j_1 < j_2$.
So, it must be that $F \cap \phi_i(c_{j_2}) \neq \varnothing$.
Take $b \in F \cap \phi_i(c_{j_2})$ and let $c_k$ be the coatom obtained by negating position $b$ in $c_{j_2}$.
Futhermore, let $d$ be the sign vector obtained by making position $b$ equal to $0$ in $c_{j_2}$.
If $b < i$, then $\phi_i(c_k) \subseteq (\phi_i(c_{j_2}) \setminus \{b\}) \cup \{b-1\}$.
If $b > i$, then $\phi_i(c_k) \subseteq (\phi_i(c_{j_2}) \setminus \{b\}) \cup \{b+1\}$.
In any case $\phi_i(c_k) <_i \phi_i(c_{j_2})$ and so $k < j_2$.
Also, $d$ is a coatom in $[e_i, c_{j_2}]$ and $x \leq d < c_k$.
Therefore (RCO2) holds and the lemma is proven.
\end{proof}

\begin{lemma}
For any $0 \leq m \leq n-1$ the simplicial complex $\Delta_{n,m}$ is a normal $(n-1)$-pseudomanifold.
\label{lem:pseudo}
\end{lemma}

\begin{proof}
Any maximal chain in $P_{n,m}$ will have length $n-1$.
So, $\Delta_{n,m}$ is a pure $(n-1)$-dimensional simplicial complex and (NP1) is verified.

Any $(n-2)$-dimensional face in $\Delta_{n,m}$ corresponds to a chain of length $n-2$ in $P_{n,m}$.
Such a chain of length $n-2$ will be missing an element of a single weight $r$ for some $1 \leq r \leq n$.
If $1 \leq r < n$, then there are exactly two ways to complete $C$ to a maximal chain since by Lemma~\ref{lem:simplicialPoset} the poset $P_{n,m} \cup \{\hat{0}\}$ is a simplicial poset.
When $r = n$, we find there are at most two ways to complete $C$ to a maximal chain by considering the sign vector of weight $n-1$ in $C$ which is covered by at most two elements.
Hence, (NP2) is verified.

To verify (NP3) we may appeal to Lemma~\ref{lem:shellable} to conclude the necessary links are connected since they are pure shellable complexes of dimension at least $1$.
Thus, the proof of the lemma is complete.
\end{proof}

\begin{theorem}
For any $0 \leq m \leq n-1$ the simplicial complex $\Delta_{n,m}$ is a combinatorial $(n-1)$-manifold.
Thus, $\Gr^{\var \leq m}_{1,n}$ is a PL manifold.
\label{thm:manifold}
\end{theorem}

\begin{proof}
If $F \in \Delta_{n,m}$ with $\dim(F) = n-2$, then $\|\link_{\Delta_{n,m}}(F)\|$ is either a single point or two discrete points.
This means we have either a $0$-dimensional PL ball or PL sphere.
If $F \in \Delta_{n,m}$ with $\dim(F) < n-2$, then $\link_{\Delta_{n,m}}(F)$ is a pseudomanifold by Lemma~\ref{lem:pseudo}.
Also, $\link_{\Delta_{n,m}}(F)$ is shellable by Lemma~\ref{lem:shellable}.
It follows that  $\|\link_{\Delta_{n,m}}(F)\|$ is a PL ball or PL sphere of the appropriate dimension.
Therefore, $\Delta_{n,m}$ is a combinatorial $(n-1)$-manifold and its geometric realization $\Gr^{\var \leq m}_{1,n}$ is a PL manifold.
\end{proof}

\subsection{Homotopy type}

In this subsection we will show that $\Gr^{\var \leq m}_{1,n}$ is homotopy equivalent to $\mathbb{RP}^m$.
Our main tool to do this is Forman's discrete Morse theory~\cite{Forman}.
We will make use of Chari's approarch to discrete Morse theory using matchings in the Hasse diagram of the closure poset of a regular CW complex~\cite{Chari}.
So, we assume throughout that all CW complexes are regular.
For any poset $P$ we consider its Hasse diagram as a directed graph with directed edge $b \to a$ for each cover relation $a \lessdot b$.
An \emph{acyclic (perfect) matching} is a (perfect) matching of the Hasse diagram such that the directed graph obtain by reversing the orientation of each edge in the matching is a directed acyclic graph.

In a CW complex we use $\sigma^{(d)}$ denote a $d$-dimensional cell.
Let $f$ be a function which assigns a real number to each cell in a CW complex.
The function $f$ is a \emph{discrete Morse function} provided both
\begin{align*}
|\{\tau^{(d-1)} \subset \overline{\sigma^{(d)}} : f(\tau^{(d-1)} ) \geq f(\sigma^{(d)}) \} | &\leq 1\\
|\{\sigma^{(d)} \subset \overline{\tau^{(d+1)}}: f(\tau^{(d+1)} ) \leq f(\sigma^{(d)}) \} | &\leq 1
\end{align*}
for every $d$-cell $\sigma^{(d)}$.
When both these cardinalities are zero the $d$-cell $\sigma^{(d)}$ is called a \emph{critical cell}.
At most one of these cardinalities can be nonzero for a given cell, and hence a discrete Morse function defines a (partial) matching of the cells.
Moreover, any matching which is an acyclic matching in the closure poset arises from some discrete Morse function.
More generally, whenever one has an acyclic matching in a poset there exists a function which is decreasing along the matching edges and increasing along the rest of the poset.
To see this one may reverse the matching edges to obtain a directed acyclic graph.
Then treating this directed acyclic graph as a poset consider a linear extension.
The desired function can be gotten by assigning to each element of the original poset its position in this chosen linear extension.

Given a CW complex $X$ and discrete Morse function $f$, For any $c \in \R$ we set 
\[X(c) =\bigcup_{\substack{ \sigma \in X \\ f(\sigma) \leq c}} \bigcup_{\tau \leq \sigma} \tau.\]
If $a < b$ are real numbers such that $[a,b]$ contains no critical values of $f$, then $X(b) \searrow X(a)$~\cite[Theorem 3.3]{Forman}.
Here $\searrow$ denotes \emph{collapsing} which gives what is known as a simple homotopy equivalence.

\begin{lemma}
For any $0 \leq m < n-1$ there exists an acyclic perfect matching of $\{x \in P_{n,m} : \sgn(e_1) \leq x\}$.
\label{lem:matching}
\end{lemma}
\begin{proof}
Let $Q = \{x \in P_{n,m} : \sgn(e_1) \leq x\}$.
We will define $\phi:Q \to Q$.
Take any $ \omega = (\omega_1, \omega_2, \dots, \omega_n) \in Q$ and let $j = \max \{i : \var((\omega_1, \omega_2, \dots, \omega_i)) = i-1\}$.
Thus, the quantity $j$ gives the length of the maximal prefix of the form $(+,-,+,-, \cdots)$.
Since $\omega \in Q$, and $m < n-1$ it must be that $j < n$.
We define $\phi(\omega) = (\omega'_1, \omega'_2, \dots, \omega'_n)$ where $\omega'_i = \omega_i$ for $i \neq j+1$ and
\[\omega'_{j+1} = \begin{cases} \omega_j & \omega_{j+1} = 0;\\ 0 & \omega_{j+1} \neq 0. \end{cases}\]
We see that $\var(\omega) = \var(\phi(\omega))$ and $\phi(\phi(\omega)) = \omega$.
Also, either $\omega \lessdot \phi(\omega)$ or $\phi(\omega) \lessdot \omega$.
It follows that $\{\{\omega, \phi(\omega)\} : \omega \in Q\}$ is a perfect matching of the Hasse diagram of $Q$.
An example of the matching produced in the case $n=3$ and $m=1$ is shown in Figure~\ref{fig:acyclic}. 
It remains to check it is acyclic.

Consider walks in the Hasse diagram with orientations on perfect matching edges reversed.
That is, walks which go up in the partial order on matching edges and go down on edges not in the matching.
Take some $ \omega = (\omega_1, \omega_2, \dots, \omega_n) \in Q$ and again let $j = \max \{i : \var((\omega_1, \omega_2, \dots, \omega_i)) = i-1\}$.
This quantity $j$ measures the maximal prefix of the form $(+,-,+,-, \cdots)$, and the length of this prefix cannot become any longer moving along directed paths.
Let $\tau = (\tau_1, \tau_2, \dots, \tau_n)$ be obtained from taking a down step from $\omega$.
This means that there exists $k \neq j+1$ such that $\tau_i = \omega_i$ for $i \neq k$ and $\tau_k = 0$ while $\omega_k \neq 0$.
If $k \leq j$, then $k-1 = \max \{i : \var((\tau_1, \tau_2, \dots, \tau_i)) = i-1\}$.
In this case we can never return to $\omega$ along a directed path since the length of our maximal prefix of the form $(+,-,+,-, \cdots)$ has gotten strictly smaller.
If $k > j+1$, then $\tau_k =0$ while $\omega_k \neq 0$, but the $k$th entry will remain zero along any directed path.
Therefore we have an acyclic perfect matching of $Q$.
\end{proof}

\begin{figure}
    \centering
    \begin{tikzpicture}
    \node (a) at (-3,2) {$(+,+,+)$};
    \node (b) at (-1,2) {$(+,+,-)$};
    \node (d) at (3,2) {$(+,-,-)$};
    
    \node (e) at (-5,0) {$(+,+,0)$};
    \node (f) at (-3,0) {$(+,0,+)$};
    \node (h) at (1,0) {$(+,-,0)$};
    \node (i) at (3,0) {$(+,0,-)$};
    
    \node (k) at (-2,-2) {$(+,0,0)$};
    
    \draw (a) to (e);
    \draw[line width=3pt,red] (a) to (f);
    
    \draw[line width=3pt,red] (b) to (i);
    \draw (b) to (e);
    
    \draw[line width=3pt,red] (d) to (h);
    \draw (d) to (i);

    \draw[line width=3pt,red] (k) to (e);
    \draw (k) to (f);
    \draw (k) to (h);
    \draw (k) to (i);

    \end{tikzpicture}
    \caption{An acyclic perfect matching.}
    \label{fig:acyclic}
\end{figure}
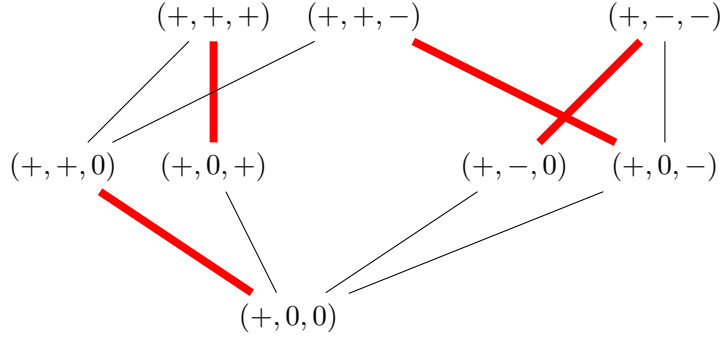

\begin{theorem}
For any $0 \leq m \leq n-1$ the space $\Gr^{\var \leq m}_{1,n} \searrow \mathbb{RP}^m$.
\label{thm:collaspe}
\end{theorem}
\begin{proof}
We can identify $\Gr^{\var \leq m}_{1,n-1}$ with $\{[x_1,x_2, \dots, x_n] \in \Gr^{\var \leq m}_{1,n}: x_1 = 0\}$.
We will show that  $\Gr^{\var \leq m}_{1,n} \searrow \Gr^{\var \leq m}_{1,n-1}$, and the theorem will follow by iteration.
Let $Q = \{x \in P_{n,m} : \sgn(e_1) \leq x\}$.
By Lemma~\ref{lem:matching} it follows we can find an acyclic perfect matching of $Q$.
We can then find a function $g: Q \to \R$  decreasing along the matching edges and increasing along the rest of $Q$.
Next define $f:P_{n,m} \to \R$ such that
\[f(\omega) = \begin{cases} g(\omega) + C & \omega \in Q \\ \wt(\omega) - 1 & \omega \not\in Q \end{cases}\]
where $C \in \R$ is such that $g(\omega) + C > n$ for all $\omega \in Q$.
We see that $f$ is a discrete Morse function for $X = \Gr^{\var \leq m}_{1,n}$.
Furthermore, $X(n) = \Gr^{\var \leq m}_{1,n-1}$ and there are no critical values in $[n, \infty)$.
Therefore, $\Gr^{\var \leq m}_{1,n} \searrow \Gr^{\var \leq m}_{1,n-1}$.
Iterating we find that $\Gr^{\var \leq m}_{1,n} \searrow \Gr^{\var \leq m}_{1,m+1}$ and $\Gr^{\var \leq m}_{1,m+1} = \Gr_{1,m+1} = \mathbb{RP}^m$.
\end{proof}

Since the homology of real projective space is known, an immediate consequence of the above theorem is that
\begin{equation}
 H_i(\Gr^{\var \leq m}_{1,n}, \k) = \begin{cases} \k & i = 0\\ \k & i = m \text{ and } m \text{ is odd } \\ 0 & \text{otherwise} \end{cases}
 \label{eq:homology}
\end{equation}
for any field $\k$ such that $\chark \neq 2$.

\subsection{Cohen-Macaulayness}

In this subsection we will investigate the Cohen-Macaulayness of the simplicial complexes $\Delta_{n,m}$.
For a reference on Cohen-Macaulay simplicial complexes we recommend~\cite{COAC}.
Let $\k$ be a field.
The simplicial complex $\Delta$ is \emph{Cohen-Macaulay over $\k$} if for all $F \in \Delta \cup \{\varnothing\}$
\[\tilde{H}_i(\link_{\Delta}(F), \k) = 0\]
for $0 \leq i \leq \dim \link_{\Delta}(F) - 1$.
We will abbreviate Cohen-Macaulay over $\k$ by $\k$-CM.
A complex being $\k$-CM depends only on the field $\k$ and the geometric realization (not on the particular triangulation).
Thus, $\Delta$ is $\k$-CM for any $\k$ whenever $\| \Delta \|$ is a ball or sphere.
So, being $\k$-CM can be thought of as a generalization of being a ball or sphere.
Also, a complex being $\k$-CM is equivalent to the corresponding Stanley-Reinser face ring being Cohen-Macaulay.
Our definition of $\k$-CM  is sometimes referred to as \emph{Reisner's criterion}~\cite{Reisner}.

\begin{corollary}
Let $\k$ be a field with $\chark \neq 2$ and $0 \leq m \leq n-1$.
The simplicial complex $\Delta_{n,m}$ is $\k$-CM if and only if $m$ is even or $m = n-1$.
\label{cor:CM}
\end{corollary}
\begin{proof}
By Theorem~\ref{thm:manifold} we know that $\| \link_{\Delta_{n,m}}(F)\|$ is a sphere or ball for all $F \in \Delta_{n,m}$.
Hence, the condition of homology to be $\k$-CM is satisfied for all $F \in \Delta_{n,m}$.
By Theorem~\ref{thm:collaspe} and its consequence Equation~(\ref{eq:homology}) we find that $\tilde{H}_i(\Delta_{n,m}, \k) = 0$ for all $0 \leq i \leq n-2$ if and only if $m$ is even or $m = n-1$.
\end{proof}

\section{Networks with one source}
\label{sec:meas}
We now turn our attention to networks with a single source.
We will show that boundary measurement Grassmannians for networks with a single source are bounded sign variation Grassmannians.
Here the boundary measurement matrix $B(N,T)$ will have a single row.
Hence, we can identify $B(N,T)$ with a vector in $\mathbb{R}^n$ or with an element of $\mathbb{RP}^{n-1}$.

\begin{figure}
    \centering
    \begin{tikzpicture}
    \draw[thick] (-3,0) circle (2cm);
    \node[draw,circle, scale=0.5] (u) at (-3.75,0) {};
    \draw[-{latex}] (-5,0) to (u);
    \draw[-{latex}] (u) to (-3,1);
    \draw[-{latex}] (u) to (-3,-1);
    \node at (-4.45,0.2) {$x_e$};
    \node at (-3.55,0.73) {$x_f$};
    \node at (-3.55,-0.73) {$x_g$};
    \node at (-5.25,0) {$k$};
    
    \draw[thick] (3,0) circle (2cm);
    \draw[-{latex}] (1,0) to (3,1);
    \draw[-{latex}] (1,0) to (3,-1);
    \node at (2.1,0.85) {$x_ex_f$};
    \node at (2.1,-0.85) {$x_ex_g$};
    \node at (0.75,0) {$k$};
    
     \draw[thick] (-3,-6) circle (2cm);
    \node[draw,circle,fill=black, scale=0.5] (u) at (-3.75,-6) {};
    \draw[-{latex}] (-5,-6) to[bend left = 60] (u);
    \draw[-{latex}] (-5,-6) to[bend right = 60] (u);
    \draw[-{latex}] (u) to (-2.5,-6);
    \node at (-4.4,-5.4) {$x_e$};
    \node at (-4.4,-6.6) {$x_f$};
    \node at (-2.8,-5.8) {$x_g$};
    \node at (-5.25,-6) {$k$};
    
    \draw[thick] (3,-6) circle (2cm);
    \draw[-{latex}] (1,-6) to (3,-6);
    \node at (2,-5.7) {$(x_e + x_f)x_g$};
    \node at (0.75,-6) {$k$};
    
    \draw[thick] (-3,-12) circle (2cm);
    \node[draw,circle, fill=black, scale=0.5] (u) at (-3.75,-12) {};
    \draw[-{latex}] (-5,-12) to (u);
    \draw[-{latex}] (-3,-11) to (u);
    \draw[-{latex}] (u) to (-3,-13);
    \node at (-4.45,-11.8) {$x_e$};
    \node at (-3.55,-11.27) {$x_f$};
    \node at (-3.55,-12.73) {$x_g$};
    \node at (-5.25,-12) {$k$};
    
    \draw[thick] (3,-12) circle (2cm);
    \draw[-{latex}] (3,-11) to (1.135,-11.25);
    \draw[-{latex}] (1,-12) to (3,-13);
    \node at (2.1,-10.85) {$x_f/x_e$};
    \node at (2.1,-12.85) {$x_ex_g$};
    \node at (0.75,-12) {$k$};
    \node at (1,-11.25) {$k'$};

    \end{tikzpicture}
    \caption{Reductions used in the proof of Lemma~\ref{lem:nointerior}.}
    \label{fig:reduction}
\end{figure}
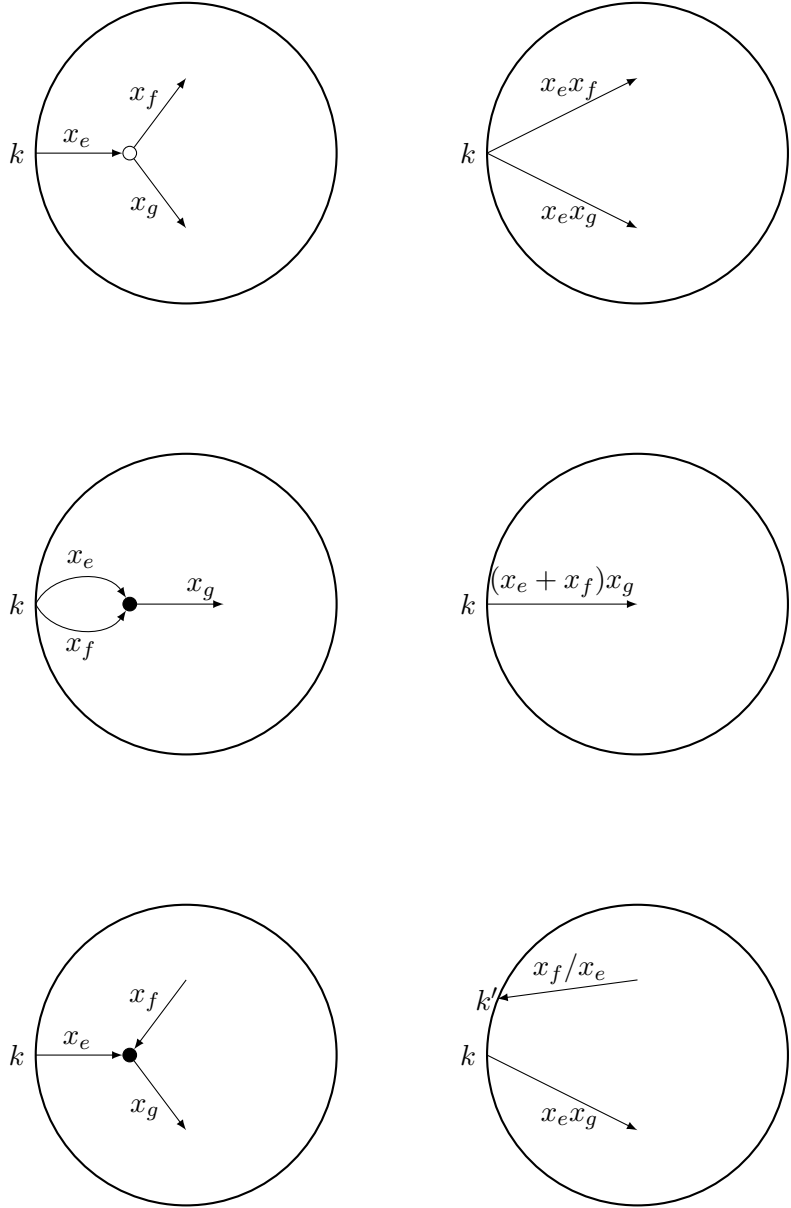

\begin{lemma}
If $(S,T)$ is a surface with chosen cuts and $N \in \Net^S_{1,n}$, then for some $n'$ there exists a directed network $\tilde{N} \in \Net^S_{1,n'}$ such that
\[\var(B(N,T)) = \var(B(\tilde{N},T))\]
and $\tilde{N}$ has no interior vertices.
\label{lem:nointerior}
\end{lemma}
\begin{proof}
We may assume $N$ has at least one interior vertex, otherwise we may take $\tilde{N} = N$.
It suffices to produce a directed network $N'$ such that $\var(N) = \var(N')$ and $N'$ has one fewer interior vertex than $N$ as the lemma will then follow by induction.
Let $k$ be the unique source of $N$.
We will consider cases depending on how edges incident on $k$ are connected to the rest of $N$.
Figure~\ref{fig:reduction} shows locally the reductions from $N$ to $N'$ in each case.

First we consider the case $e = (k,u)$ is an edge in $N$ where $u$ is a white vertex.
Let $f = (u,v)$ and $g = (u,w)$ be the other two edges incident on $u$.
To construct $N'$ we remove the vertex $u$ as well as the edges $e, f$, and $g$.
We then add the edge $f'=(k,v)$ with weight $x_{f'} = x_e x_f$ and the edge $g'=(k,w)$ with weight $x_{g'} = x_e x_g$.
In this case $B(N,T) = B(N',T)$.
So, $\var(B(N,T)) = \var(B(N',T))$.

Next we consider the case when $e = (k,u)$ and $f = (k,u)$ are edges in $N$ terminating at a black vertex $u$.
Let $g = (u,v)$ be the other edge incident on $u$.
We construct $N'$ by first removing the vertex $u$ along with the edges $e, f$, and $g$.
We then add the edge $e' = (k,v)$ and with weight $x_{e'} = (x_e + x_f)x_g$.
In this case we also find $B(N,T) = B(N',T')$.
So, $\var(B(N,T)) = \var(B(N',T'))$ again.

Lastly, we consider the case where $e = (k,u)$ is an edge in $N$ incident on a black vertex $u$ while $f=(v,u)$ and $g=(u,w)$ are the other edges incident on $u$ with $v \neq k$.
In this case we construct $N'$ by removing $u$ as well as the edges $e, f$, and $g$.
We then add a new sink $k'$ to the boundary of $S$ near $k$ as shown in Figure~\ref{fig:reduction}.
We also add edges $e' = (k,w)$ and $f' = (v,k')$ with weights $x_{e'} = x_ex_g$ and $x_{f'} = x_f/x_e$.
If $N \in \Net^S_{1,n}$, then $N' \in \Net^S_{1,n+1}$.
Let $I$ and $I' = I \cup \{k'\}$ be the set of boundary vertices of $N$ and $N'$ respectively.
We let $B(N,T) = [b_i : i \in I]$ and $B(N',T) = [c_i : i \in \tilde{I}]$.
Since $k$ is a source $b_k = c_k = 1 > 0$.
There are no sources of $N'$ strictly between $k$ and $k'$.
Also $k$ and $k'$ are on the same boundary component, and it follows that $c_{k'} \geq 0$.
For any $i \in I$ we can enumerate paths $k \pathto i$ by considering how many times the path traverses the edge $f$.
Notice that when a path traverses the edge $f$ it arrives at the vertex $u$, and the path must leave the vertex $u$ on the edge $g$.
Hence, for all $i \in I \setminus \{k\}$,
\[b_i = c_i - c_ic_{k'} + c_ic_{k'}^2 - \cdots = \frac{c_i}{1 + c_{k'}}\]
where the negative signs come from the rotation number increasing for a path which includes a cycle through $u$ in $N$ which is not present in $N'$.
In particular, we find that $\sgn(b_i) = \sgn(c_i)$.
So, we can conclude that $\var(B(N,T)) = \var(B(N',T))$ because $b_i$ and $c_i$ have the same sign for each $i \in I$.
\end{proof}

Given any surface with chosen cuts $(S,T)$, we have an abstract graph $\mathcal{G}_{(S,T)}$ whose vertices are the boundary components of $S$ and whose edges consist of pairs of boundary components connected by a cut.

\begin{figure}
    \centering
    \begin{tikzpicture}
    \draw[thick] (-6,1) -- (6,1);
    \draw[thick] (-6,0) -- (6,0);
    \draw[dotted] (-6,1)--(-6,0);
    \draw[dotted] (-2,1)--(-2,0);
    \draw[dotted] (2,1)--(2,0);
    \draw[dotted] (6,1)--(6,0);
    
    \node at (0,-0.15) {$k$};
    \node at (-1,1.15) {$x$};
    \node at (0,1.15) {$y$};
    \node at (1,1.15) {$z$};
    
    \node at (4,-0.15) {$k$};
    \node at (3,1.15) {$x$};
    \node at (4,1.15) {$y$};
    \node at (5,1.15) {$z$};
    
     \node at (-4,-0.15) {$k$};
    \node at (-5,1.15) {$x$};
    \node at (-4,1.15) {$y$};
    \node at (-3,1.15) {$z$};
    
    \draw (0,0) -- (3,1);
    \draw (0,0) -- (5,1);
    \draw (4,0) -- (6,0.4);
    \draw(4,0) -- (6,0.75);
    \draw (-4,0) -- (-1,1);
    \draw (-4,0) -- (1,1);
    \draw (-6,0.4) -- (-3,1);
    \draw (-6,0.75) -- (-5,1);

    \draw[thick] (0,-5) circle (3cm);
    \draw[thick] (0,-5) circle (1cm);
    \draw[dotted] (0,-2) to (0,-4);

    \node at (-0.45,-5.65) {$z$};
    \node at (-0.85,-5) {$y$};
    \node at (-0.5,-4.3) {$x$};
    \node at (0,-8.15) {$k$};

    \draw (0,-8) to[bend right = 90] (0,-3);
     \draw (0,-8) to[bend right = 90] (0,-2.5);
    \draw (0,-3) to[bend right = 60] (-0.65,-4.25);
    \draw (0,-2.5) to[bend right = 45] (-2,-5);
     \draw (-2,-5) to[bend right = 60] (-0.6,-5.8);

    \end{tikzpicture}
    \caption{A situation illustrating part of the proof of Lemma~\ref{lem:onesourcebound}.}
    \label{fig:universal_cover}
\end{figure}
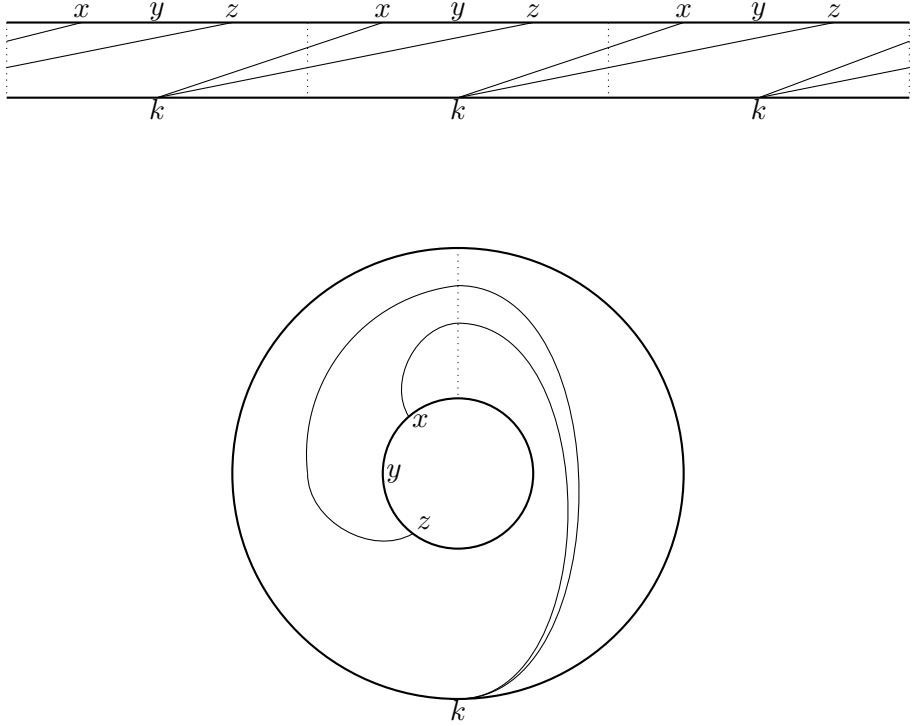

\begin{lemma}
If $(S,T)$ is a surface with chosen cuts, then $\mathcal{G}_{(S,T)}$ is a tree.
\label{lem:tree}
\end{lemma}
\begin{proof}
By definition $T$ must be simply connected.
If $\mathcal{G}_{(S,T)}$ is not connected, then there would exists loops in $T$ which cannot be contracted.
If $\mathcal{G}_{(S,T)}$ contains a cycle, then $T$ would not be connected.
Therefore, $\mathcal{G}_{(S,T)}$ must be a tree. 
\end{proof}

\begin{lemma}
If $(S,T)$ is a surface with chosen cuts with $b > 0$ boundary components, then 
\[\var(B(N,T)) \leq \min\{2(b-1), n-1\}\]
for any $N \in \Net^S_{1,n}$ and choice of positive real weights.
\label{lem:onesourcebound}
\end{lemma}
\begin{proof}
We need only show that $\var(B(N,T)) \leq 2(b-1)$ since $\var(B(N,T)) \leq n-1$ for $N \in \Net^S_{1,n}$ on any surface.
Furthermore, by Lemma~\ref{lem:nointerior} it suffices to consider networks with no interior vertices.
Take $N \in \Net^S_{1,n}$ with no interior vertices and let $k$ be the unique boundary source of $N$.
In particular, the network $N$ will have no cycles.
Then $B(N,T) = [B_i : i \in I]$ where 
\[B_i = \sum_{P: k \pathto i} (-1)^{s_{ki} + r_P + 1} \wt(P) = \sum_{P: k \pathto i} (-1)^{r_P + 1} \wt(P)\]
since $s_{ki} = 0$ for all $i \in I$.
Thus, signs arising in $B(N,T)$ come only from the rotation number.
Since $N$ is acyclic the rotation number $r_P$ can only be changed by a path $P: k \pathto i$ crossing cuts used to connected $i$ and $k$ along the boundary of $T$.

Assume we have boundary sinks $x < y < z $ such that $x$, $y$, and $z$ are all on the same component of the boundary.
If $x$, $y$, and $z$ are on the same boundary component as $k$, then $B_x, B_y, B_z \geq 0$.
So, we will assume that $k$ is on a different boundary component.
We claim we can compute $B_x$, $B_y$ and $B_z$ with a network on the annulus.
Let $C$ be the boundary component containing $k$ and $C'$ be the boundary component containing $x$, $y$, and $z$.
First, delete all vertices except for $k$, $x$, $y$, and $z$.
Next we contract all boundary components except $C$ and $C'$ to a point.
By Lemma~\ref{lem:tree} $\mathcal{G}_{(S,T)}$ is a tree.
We then keep only the cuts corresponding to edges along the unique path between $C$ and $C'$ in $\mathcal{G}_{(S,T)}$.
It then follows that  $B_x$, $B_y$ and $B_z$ can be computed on the new network we have just formed on the annulus.
Any cut not along the unique path between $C$ and $C'$ in $\mathcal{G}_{(S,T)}$ will be traversed either zero or two times when closing paths, and hence these cuts will not contribute the parity of the rotation number.

We continue working with the network we have formed on the annulus.
Further assume that $B_xB_z > 0$.
It follows that $B_xB_y \geq 0$ and $B_yB_z \geq 0$.
In Figure~\ref{fig:universal_cover} we see an example of this situation depicted on both the annulus and the universal cover of the annulus.
It then follows that $\var([B_x : x \in C]) \leq 1$.

We can now prove the lemma by induction on the number of boundary components $b$.
When $b = 0$ we know that $\var(B(N,T)) = 0$ and the lemma holds.
If $b>0$, then $S$ can be obtained by adding a boundary component $C$ to some $S'$ with $b-1$ boundary components.
By Lemma~\ref{lem:tree} a boundary component $C$ which touches exactly one cut will necessarily exist, and we choose such a $C$.
We then have a network $N'$ on $S'$ with choice of cuts $T'$ obtained by deleting all vertices, edges, and unique cut which intersects with $C$.
By induction $\var(B(N',T')) \leq 2(b-2)$.
Also, $B(N,T)$ can be obtained by taking $B(N',T')$ then inserting the entire vector $[B_x : x \in C]$ somewhere.
That is, because $C$ was a leaf in $\mathcal{G}_{(S, T)}$ the entries $B_x$ for $x \in C$ will be a continuous segment in the vector $B(N,T)$.
Thus, $\var(B(N,T)) \leq 2(b-2) + 2 = 2(b-1)$ as desired since $\var(B(N',T')) \leq 2(b-2)$ and $\var([B_j : j \in C']) \leq 1$.
\end{proof}

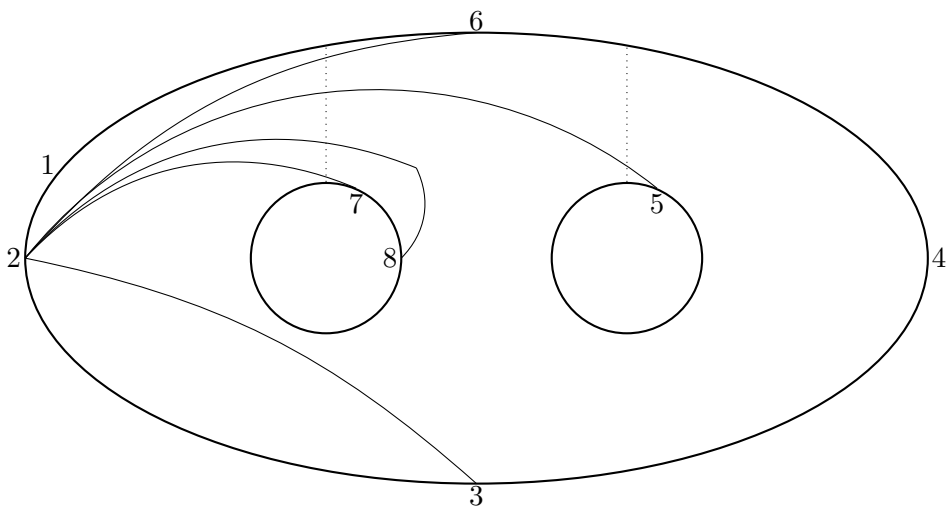
\begin{figure}
    \centering
    \begin{tikzpicture}
	\draw[thick] (0,0) ellipse (6cm and 3cm);
	\draw[thick] (-2,0) circle (1cm);
	\draw[thick] (2,0) circle (1cm);
	\draw[dotted] (-2,2.8) to (-2,1);
	\draw[dotted] (2,2.8) to (2,1);
	
	\draw (-6, 0) to[bend left=15] (0,-3);
	\draw (-6, 0) to[bend left=22] (0,3);
	\draw (-6, 0) to[bend left=45] (2.4,0.93);
	\draw (-6, 0) to[bend left=35] (-1.6,0.93);
	\draw (-6, 0) to[bend left=35] (-0.8,1.2);
	\draw (-0.8, 1.2) to[bend left=35] (-1,0);
	
	\node at (-5.7,1.24) {$1$};
	\node at (-6.15,0) {$2$};
	\node at (0,-3.16) {$3$};
	\node at (6.15,0) {$4$};
	\node at (2.4,0.73) {$5$};
	\node at (0,3.16) {$6$};
	\node at (-1.6,0.73) {$7$};
	\node at (-1.15,0) {$8$};

    \end{tikzpicture}
    \caption{A situation illustrating the construction in the proof of Lemma~\ref{lem:onesourcebound}.}
    \label{fig:construction}
\end{figure}

\begin{lemma}
If $S$ is a closed orientable surface with $b > 0$ boundary components and $v \in \mathbb{RP}^{n-1}$ with $\var(v) \leq \min\{2(b-1), n-1\}$, then 
\[B(N,T) = v\]
for some $N \in \Net^S_{1,n}$ and choice of cuts $T$.
\label{lem:onesourceconstruct}
\end{lemma}
\begin{proof}
We will make the choice of cuts where each boundary component is connected to the external boundary component by a cut.
So, $\mathcal{G}_{S,T}$ will be the star graph with a single vertex of degree $b-1$ while the remaining vertices are leaves.
Furthermore, we parameterize the boundary of $T$ so that we start on the external boundary component.
Without loss of generality we take a representative of $v \in \mathbb{RP}^{n-1}$ with $\var(v) \leq \min\{2(b-1), n-1\}$ where the first nonzero entry is $1$.
We can then realize $v$ as $B(N,T)$ where the first nonzero entry corresponds to the source of $N$.
Each positive entry of $v$ will correspond to a vertex on the external boundary component while each negative entry of $v$ corresponds to a sink on another boundary component.
Hence, when ignoring zeros we have each run of negative entries on a different boundary component.
By the assumption that $\var(v) \leq 2(b-1)$ it follows that such a construction can be made.
Figure~\ref{fig:construction} shows how a network constructed as described would look for $b = 3$ and $\sgn(v) = (0,+,+,0,-,+,-,-)$.
\end{proof}

Combining Lemma~\ref{lem:onesourcebound} and Lemma~\ref{lem:onesourceconstruct} we obtain the following theorem.

\begin{theorem}
If $S$ is a closed orientable surface of genus zero with $b > 0$ boundary components, then
\[\Gr^S_{1,n} = \Gr_{1,n}^{\var \leq \min \{2(b-1), n-1\}}.\]
\label{thm:onesource}
\end{theorem}

Theorem~\ref{thm:onesource} says that $\Gr_{1,n}^{\var \leq m}$ is a boundary measurement Grassmannian for some $S$ if and only if $m$ is even or $m = n-1$.
Thus, from Corollary~\ref{cor:CM} we see that $\Gr^S_{1,n}$ is $\k$-CM for any $\k$ not of characteristic $2$.

\bibliographystyle{alpha}
\bibliography{SignBib}
\end{document}